\documentclass{amsart}

\newcommand\tw[1]{{#1}}

\usepackage{mathptmx}
\usepackage{mathrsfs}
\usepackage{verbatim}
\usepackage{url}
\usepackage[all]{xy}
\usepackage{xcolor}
\usepackage{amsmath,amsthm}
\usepackage{amssymb}
\usepackage{enumerate}
\usepackage{graphicx}
\usepackage{epstopdf}

\newtheorem{thm}{Theorem}[section]

\newtheorem{lem}[thm]{Lemma}

\theoremstyle{definition}

\numberwithin{equation}{section}

\begin{document}

\title[Equicontinuity of minimal sets for amenable group actions on dendrites]{Equicontinuity of minimal sets for\\ amenable group actions on
dendrites}

\author{Enhui Shi\ \ \&\  \ Xiangdong Ye}

\address[E.H. Shi]{School of Mathematical Sciences, Soochow University, Suzhou 215006, P. R. China}
\email{ehshi@suda.edu.cn}

\address[X. Ye]
{Wu Wen-Tsun Key Laboratory of Mathematics, USTC, Chinese Academy of Sciences and Department of Mathematics, University of Science and
Technology of China, Hefei, Anhui 230026, China}
\email{yexd@ustc.edu.cn}


\begin{abstract}
We show that if $G$ is an amenable group acting on a dendrite $X$,
then the restriction of $G$ to any  minimal set $K$ is equicontinuous, and $K$ is either finite or homeomorphic to the Cantor set.
\end{abstract}
\maketitle

\section{Introduction}
It is well known that every \tw{continuous} action of
\tw{a topological group}~$G$ on \tw{a compact metric space}~$X$
must have a minimal set~$K$. A natural
question is \tw{to ask}
what can \tw{be said} about the topology of~$K$,
and the dynamics of the subsystem~$(K, G)$.
The answer
to this question \tw{certainly}
depends on the topology of~$X$ and
\tw{involves} the algebraic structure of~$G$.
\tw{We assume throughout that groups are
topological groups, and that the actions
are continuous.}

In the case of
\tw{an orientation-preserving} group action on
the circle~$\mathbb S^1$, the topology of minimal sets and the dynamics on them are well understood.
In fact, for any
action of
\tw{a topological} group~$G$ on~$\mathbb S^1$,
the minimal set~$K$ can only be a finite set,
a Cantor set, or the whole circle
(see, \tw{for example,}~\cite{Nav}).
\tw{The interaction between the topology of~$K$
and the algebraic structure of~$G$ arises as follows.}
\begin{itemize}
\item \tw{If}~$K$ is a Cantor set, then~$(K, G)$ is
\tw{semi-conjugate} to a minimal action \tw{of~$G$}
on~$\mathbb S^1$.
\item If~$K=\mathbb S^1$, then~$(K, G)$ is either
equicontinuous,
or~$(K,G)$ is~$\epsilon$-strongly proximal for
some~$\epsilon>0$, and~$G$ contains a
free non-commutative subgroup
(\tw{so, in particular},~$G$ cannot be
amenable; see \cite{Ma}).
\end{itemize}
The
classes of minimal group actions on the circle
\tw{up to topological conjugacy}
\tw{have been} classified by
Ghys using bounded Euler class (see~\cite{Ghy, Gh}).

Recently, there
\tw{has been} considerable progress
in \tw{the study of} group actions on dendrites.
Minimal group actions on dendrites
appear naturally in the theory of~$3$-dimensional
hyperbolic geometry (see,
for example,~\cite{Bo, Mi}).
Shi proved that every minimal group action
on \tw{a dendrite} is strongly proximal,
and the acting group cannot be amenable (see~\cite{Sh, SWZ}).
Based on the results obtained by
Marzougui and Naghmouchi in~\cite{MN},
Shi and Ye showed that
\tw{an} amenable
group action on
\tw{a dendrite} always
has a minimal set consisting of~$1$
or~$2$ points (see~\cite{SY}),
which is also implied by
the work of Malyutin and Duchesne--Monod
(see~\cite{Mal, DM}). For group actions on dendrites with no finite orbits, Glasner and Megrelishvili showed the
extreme proximality of minimal subsystems
and the strong proximality of the whole system; for amenable group actions on dendrites,
they showed that every infinite minimal subsystem is almost automorphic (see~\cite{GM}).
For~$\mathbb Z$ actions on dendrites,
Naghmouchi proved that every minimal set is either
finite or an adding machine (see~\cite{Nag}).

We \tw{prove} the following theorem in this
paper, which extends the corresponding
result for~$\mathbb Z$ actions in~\cite{Nag},
and
\tw{answers} a question proposed by
Glasner and Megrelishvili in~\cite{GM}.

\begin{thm}\label{theoremwas1-1}
Let~$G$ be an amenable group acting on a
dendrite~$X$,
\tw{and suppose that}~$K$ is a minimal set
\tw{for the action}.
Then~$(K, G)$ is equicontinuous,
and~$K$ is either finite or homeomorphic to the Cantor set.
\end{thm}

Recently, Shi and Ye have shown that every amenable group action on uniquely arcwise connected continua (without the assumption of local connectedness)
must have a minimal set consisting of~$1$ or~$2$
points (see~\cite{SY1}).
We end this
\tw{introduction} with the following general question:

\begin{quote}
{\it What results holding for group actions on dendrites can be extended to actions on uniquely arcwise connected continua?}
\end{quote}

In the following, we assume all the groups
\tw{appearing}
in this paper are countable.

\section{Preliminaries}
\subsection{Group actions}
Let~$X$ be a compact metric space,~${\rm Homeo}(X)$
\tw{its homeomorphism group},
\tw{and let}~$G$ be a group. A group
homomorphism~$\phi: G\rightarrow {\rm Homeo}(X)$ is called an
\emph{action} of $G$ on $X$; we
\tw{also write}~$(X, G)$ to denote
\tw{an} action of~$G$
on~$X$.
For brevity, we usually
\tw{write}~$gx$ or~$g(x)$ instead of~$\phi(g)(x)$.

The \emph{orbit}
of~$x\in X$ under the action of~$G$ is the
set
\[
Gx=\{gx\mid g\in
G\}.
\]
For a subset~$A\subseteq X$,
set~$GA=\bigcup_{x\in
A}Gx$;
\tw{a set}~$A$ is said to be~$G$-\emph{invariant}
if~$GA=A$;
\tw{finally, a point}~$x\in X$ is called a
\emph{fixed point} of
\tw{the action} if~$Gx=\{x\}$.
If~$A$ is a~$G$-invariant closed subset of~$X$
and~$\overline{Gx}=A$ for every~$x\in A$
\tw{(that is, the orbit of each point is dense)},
then~$A$ is called a \emph{minimal set for the action}.
\tw{In this setting every action has a minimal
set by Zorn's lemma.}

A Borel probability measure~$\mu$ on~$X$
is called~$G$-\emph{invariant} if~$\mu(g(A))=\mu(A)$
for every Borel set~$A\subset X$ and every~$g\in G$.
The following lemma follows directly from the~$G$-invariance of \tw{the support}~${\rm supp(\mu)}$
\tw{(which is automatic)}.

\begin{lem}\label{lemmawas2-1}
If~$(X, G)$ is minimal and~$\mu$ is a~$G$-invariant Borel probability
measure on~$X$, then~${\rm supp(\mu)}=X$.
\end{lem}

\begin{lem}\label{lemmawas2-2}
\tw{Suppose that a group}~$G$ acts on a
compact metric space~$X$,
\tw{ and that}~$K$ is a minimal set in~$X$
\tw{carrying a}~$G$-invariant Borel probability
measure~$\mu$.
If~$U$ and~$V$ are open sets in~$X$ such that~$V\supset U$
and~$g(V\cap K)\subset U\cap K$ for some~$g\in G$,
then~$K\cap (V\setminus {\overline U})=\emptyset$.
\end{lem}

\begin{proof}
Assume to the contrary that there is
some~$u\in K\cap (V\setminus {\overline U})$.
Then there is
\tw{an} open neighborhood~$W\ni u$
with~$W\subset V\setminus {\overline U}$.
By Lemma~\ref{lemmawas2-1},
\tw{we have}~$\mu(W\cap K)>0$.
\tw{This then implies that}~$\mu(V\cap K)=\mu (g(V\cap K))\leq \mu(U\cap K)<\mu(V\cap K)$, a contradiction.
 \end{proof}

\subsection{Amenable groups}
Amenability was first introduced by
von Neumann. Recall that a
countable group $G$ is
\tw{said to be}
\emph{amenable} if there is a
sequence of finite sets~$F_i$ ($i=1, 2, 3,\ \dots$) such that
\[
\lim\limits_{i\to\infty}\frac{|gF_i\bigtriangleup F_i|}{|F_i|}=0
\]
for every~$g\in G$, where~$|F_i|$ is the
number of elements in~$F_i$.
The \tw{sequence}~$(F_i)$ is called a \tw{\emph{F{\o}lner sequence}}
\tw{and each~$F_i$ a F{\o}lner set}.
It is well known
that solvable groups and finite groups are amenable
\tw{ and that }any group containing a free
\tw{non-commutative} subgroup is not amenable.
One may consult
\tw{the monograph of Paterson}~\cite{Pa}
for the proofs of the following lemmas.

\begin{lem}\label{lemmawas2-3}
Every subgroup of an amenable group is amenable.
\end{lem}

\begin{lem}\label{lemmawas2-4}
A group $G$ is amenable if and only if every action of $G$ on a compact metric
space~$X$ has a~$G$-invariant Borel probability measure on~$X$.
\end{lem}

\subsection{Dendrites}
A \emph{continuum} is a
\tw{non-empty} connected compact metric space. A
continuum is \tw{said to be} \emph{non-degenerate}
if it is not a single point.  An
\emph{arc} is a continuum which is homeomorphic to the closed
interval~$[0, 1]$.
A continuum~$X$ is \emph{uniquely arcwise
connected} if for any two points~$x\not=y\in X$ there is a unique
arc~$[x, y]$ in~$X$ \tw{connecting}~$x$ and~$y$.

A \emph{dendrite}~$X$ is
a locally connected, uniquely arcwise connected, continuum.
If~$Y$ is a subcontinuum of
\tw{a dendrite}~$X$, then~$Y$ is
called a \emph{subdendrite} of~$X$.
For a dendrite~$X$ and a point~$c\in X$,
if~$X\setminus\{c\}$ is not connected,
then $c$ is called a \emph{cut point} of~$X$;
if~$X\setminus\{c\}$ has at least~$3$ components,
then~$c$ is called a \emph{branch point} of~$X$.

\tw{Lemmas~\ref{lemmawas2-5}
to~\ref{lemmawas2-8} are
taken} from~\cite{Na}.

\begin{lem}\label{lemmawas2-5}
Let~$X$ be a dendrite with metric~$d$.
Then, for every~$\epsilon>0$, there is a~$\delta>0$
such that~${\rm diam}([x, y])<\epsilon$ whenever~$d(x, y)<\delta$.
\end{lem}

\begin{lem}\label{lemmawas2-6}
Let~$X$ be a dendrite.
If~$A_i\ (i=1, 2, 3, \dots)$ is a
sequence of mutually disjoint sub-dendrites of~$X$,
then~${\rm diam}(A_i)\rightarrow 0$ as~$i\rightarrow\infty$.
\end{lem}

\begin{lem}\label{lemmawas2-7}
Let~$X$ be a dendrite. Then~$X$ has
at most countably many branch points.
If~$X$ is nondegenerate,
then the cut point set of~$X$ is uncountable.
\end{lem}

\begin{lem}\label{lemmawas2-8}
Let~$X$ be a dendrite and~$c\in X$. Then each component~$U$
of~$X\setminus\{c\}$ is open in~$X$,
and~$\overline U=U\cup \{c\}$.
\end{lem}

Now we give a proof of the following technical lemma.

\begin{lem}\label{lemmawas2-9}
Let~$X$ be a dendrite and let~$f:X\rightarrow X$ be a
homeomorphism. Suppose~$o$ is a fixed point of~$f$,
\tw{and let}~$c_1, c_2$ be cut points of~$X$
different from~$o$.
Suppose \tw{that}~$U$ is a component of~$X\setminus \{c_1\}$
not \tw{containing}~$o$,
\tw{that}~$V$ is a component
of~$X\setminus \{c_2\}$
not \tw{containing}~$o$,
\tw{and that}~$f(c_1)\in V$. Then~$f(U)\subset V$.
\end{lem}

\begin{proof}
Assume to the contrary that there is
some~$u\in U$ with~$f(u)\notin V$.
Since~$c_2$ is a cut point,~$f(c_1)\in
V$,~$o\notin V$, and~$f(o)=o$,
we have~$c_2\in [f(o), f(c_1)]$
and~$c_2\in [f(u), f(c_1)]$.
This implies
\tw{that}~$f^{-1}(c_2)\in [o, c_1]\cap [u, c_1]=\{c_1\}$
since~$o\notin U$.
Thus~$f(c_1)=c_2$, which contradicts
\tw{the assumption that}~$f(c_1)\in V$.
\end{proof}

If~$[a, b]$ is an arc in a dendrite~$X$,
denote by~$[a, b)$,~$(a,b]$,
and~$(a, b)$ the
sets~$[a,b]\setminus\{b\}$,~$[a,b]\setminus\{a\}$,
and~$[a,b]\setminus\{a, b\}$, respectively.

\subsection{Equicontinuity}

Let~$X$ be a compact metric space with
metric~$d$, and let~$G$ be a group
acting on~$X$. Two points~$x, y\in X$
are said to be \emph{regionally proximal}
if there are sequences~$(x_i)$,~$(y_i)$
in~$X$
and~$(g_i)$ in~$G$
such that~$x_i\rightarrow x$
\tw{and}~$y_i\rightarrow y$
as~$i\rightarrow\infty$,
and~$\lim g_ix_i=\lim g_iy_i=w$
for some~$w\in X$.
If~$x,y$ are regionally proximal
and~$x\not=y$, then $\{x, y\}$
\tw{is} said to be a \emph{non-trivial regionally proximal pair}.
The action~$(X, G)$ is \emph{equicontinuous}
if, for every~$\epsilon>0$, there is a~$\delta>0$ such
that~$d(gx, gy)<\epsilon$ \tw{for all~$g\in G$}
whenever~$d(x, y)<\delta$.

The following lemma can be \tw{found} in~\cite{Au}.

\begin{lem}\label{lemmawas2-10}
Suppose~$(X,G)$ is a group action.
Then~$(X,G)$ is equicontinuous if and only if
it contains no non-trivial regionally proximal pair.
\end{lem}

\section{Proof of the main theorem}
In this section we are going to show our main result. Before doing this we state two simple lemmas.

\begin{lem}\label{lemmawas3-1}
Suppose a group~$G$ acts on the closed interval~$[0, 1]$.
If~$K\subset [0, 1]$ is minimal, then~$K$ contains at most~$2$ points.
\end{lem}

\begin{proof}
Let~$x=\inf{K}$ and~$y=\sup{K}$.
Then~$G$ preserves the set~$\{x, y\}$,
so~$K=\{x,y\}$ by the minimality of~$K$.
\end{proof}

\begin{lem}[\tw{See}~\cite{SY}]\label{lemmawas3-2}
Let~$G$ be an amenable group acting
on a dendrite~$X$. Then there is a~$G$-invariant
set consisting of~$1$ or~$2$ points.
\end{lem}

Now we are ready to prove the main result.

\begin{proof}[Proof of Theorem~\ref{theoremwas1-1}] We first show that~$(K,G)$ is equicontinuous.

Assume to the contrary that~$(K, G)$ is not equicontinuous.
Then \tw{by}
Lemma~\ref{lemmawas2-10},
there are~$u\not=v\in K$ such that~$u,v$
are regionally proximal; that is,
there are sequences~$(u_i), (v_i)$
in~$X$ and~$(g_i)$ in~$G$ with
\begin{equation}\label{equationwas3-1}
u_i\rightarrow u,  v_i\rightarrow v,\ \lim g_ix_i=\lim g_iy_i=w
\end{equation}
\tw{as~$i\to\infty$}
for some~$w\in K$.

\tw{By} Lemma~\ref{lemmawas3-2},
there are~$o_1, o_2\in X$ such that~$\{o_1, o_2\}$ is
a~$G$-invariant set. Then~$[o_1, o_2]$ is~$G$-invariant
by the \tw{unique} arcwise
connectedness of~$X$. From the assumption,~$K$ is infinite
so~$K\cap [o_1, o_2]=\emptyset$ by Lemma~\ref{lemmawas3-1}.
Without loss of generality, we may
suppose \tw{that}~$o_1=o_2$
\tw{and denote this common point by}~$o$; otherwise, we need only collapse~$[o_1, o_2]$ to one point.
Then~$o$ is a fixed point \tw{for the action}.

\medskip
\noindent
{\bf Case 1.}  $[u, o]\cap [v, o]=\{o\}$ (see Fig.1(1)).
By Lemma~\ref{lemmawas2-7},
we can \tw{choose} cut points~$c_1\in (u, o)$
and~$c_2\in (v, o)$. Let~$D_u$ be the component
of~$X\setminus \{c_1\}$, which contains~$u$;
let~$D_v$ be the
component of~$X\setminus \{c_2\}$, which contains~$v$.
From minimality and Lemma~\ref{lemmawas2-8},
there is some~$g'\in G$ with~$g'w\in D_u$.
From~\eqref{equationwas3-1}
and Lemma~\ref{lemmawas2-5},
we have
\begin{equation}\label{equationwas3-2}
u_i\in D_u, v_i\in D_v \ {\mbox {and}}\ g'g_i[u_i, v_i]\subset D_u
\end{equation}
\tw{for large enough~$i$.}
Write~$g=g'g_i$.
Then~$o\in [u_i, v_i]$ and~$g(o)\in D_u$.
This is a contradiction, since~$o$ is fixed by~$G$.

\medskip
\noindent {\bf Case 2.} $[u, o]\cap [v, o]=[z, o]$
for some~$z\not=o$.

\medskip

\noindent {\bf Subcase 2.1.} $z=v$ (see Fig.1(2)).
Then~$u\not=z$ and~$z\in K$.
Take a cut point~$c_1\in (u, z)$
\tw{and let}~$D_u$ be the
component of~$X\setminus \{c_1\}$
which contains~$u$.
Then~$v\notin D_u$,
and there is some~$g\in G$
with~$gz\in D_u$ by the minimality of~$K$.
Take a cut point~$c_2\in (z, o)$ which is sufficiently
close to~$z$
\tw{to ensure} that~$g(c_2)\in D_u$.
Let~$D_z$ be the component
of~$X\setminus \{c_2\}$ which contains~$z$.
By Lemma~\ref{lemmawas2-4},
there is
a~$G$-invariant Borel probability measure on~$K$.
Applying Lemma~\ref{lemmawas2-9},
we get~$g(D_z)\subset D_u$,
\tw{which} contradicts Lemma~\ref{lemmawas2-2},
since~$z\in D_z\setminus {\overline D_u}$.

\medskip

\noindent {\bf Subcase 2.2.} $z=u$. \tw{In this case
we can deduce a contradiction
along the lines of the} argument in Subcase~2.1.

\medskip

\noindent {\bf Subcase 2.3.} $z\not=u$
and~$z\not=v$ (see Fig.1(3)).
Take a cut point~$c_1\in (u, z)$. Let~$D_u$
be the component of~$X\setminus \{c_1\}$,
which contains~$u$.
Similar to the argument in Case~1,
there is some~$g\in G$
with~$g(z)\in D_u$.
Take a cut point~$c_2\in (z, o)$
which is sufficiently close to~$z$
\tw{to ensure} that~$g(c_2)\in D_u$.
Let~$D_z$
be the component of~$X\setminus \{c_2\}$,
which contains~$z$. Then~$g(D_z)\subset D_u$ by
Lemma~\ref{lemmawas2-9}.
This contradicts Lemma~\ref{lemmawas2-2}
since~$v\in D_z\setminus {\overline D_u}$.

\medskip

Now we prove that if~$K$ is not finite,
then~$K$ is homeomorphic to the Cantor set.
\tw{If not, then} there is some
non-degenerate connected
component~$Y$ of~$K$.
Clearly, for any~$g, g'\in G$,
either~$g(Y)=g'(Y)$ or~$g(Y)\cap g'(Y)=\emptyset$.
This, together with Lemma~\ref{lemmawas2-6}
and the equicontinuity of~$(K, G)$,
implies that the subgroup~$H=\{g\in G: g(Y)=Y\}$
has finite index in~$G$.
\tw{It follows that}~$(Y, H)$ is minimal.
This contradicts Lemma~\ref{lemmawas3-2}
and Lemma~\ref{lemmawas2-3}, since~$Y$
is a non-degenerate dendrite.
\end{proof}

\begin{figure}[htbp]
\centering
\includegraphics[scale=0.8]{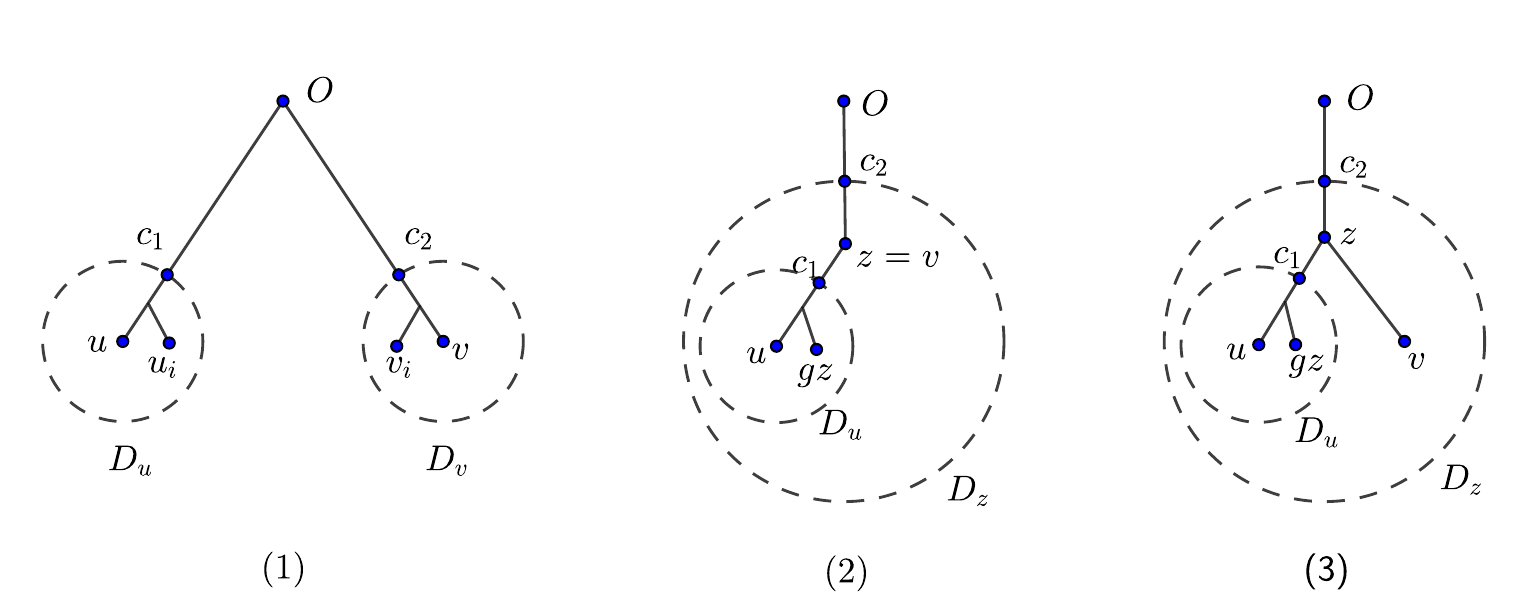}
\centerline{Fig. 1}
\end{figure}

\subsection*{Acknowledgements}
The authors would like to thank Eli Glasner for sending us the early version of his work with  Megrelishvili.
The work is supported by NSFC (No. 11771318, 11790274, 11431012).


\end{document}